\documentclass{amsart}
\usepackage{amsmath}
\usepackage{graphicx}
\usepackage{latexsym}
\usepackage{color}
\usepackage{amscd}
\usepackage[all]{xy}
\usepackage{enumerate}
\usepackage{hyperref}
\usepackage{subfigure}
\usepackage{soul}
\usepackage{comment}
\newtheorem{thm}{Theorem}

\newtheorem{prop}[thm]{Proposition}

\newtheorem{lemma}[thm]{Lemma}

\newtheorem{proposition}[thm]{Proposition}
\newtheorem{question}[thm]{Question}

\newtheorem*{mthma}{Theorem~A}
\newtheorem*{mthmb}{Theorem~B}
\newtheorem*{mthmc}{Theorem~C}

\theoremstyle{definition}

\newtheorem*{definition*}{Definition}

\newtheorem{remark}[thm]{Remark}

\newcommand{\CPb}{\overline{\mathbb{CP}}{}^{2}}
\newcommand{\CP}{{\mathbb{CP}}{}^{2}}

\newcommand{\Q}{\mathbb{Q}}
\newcommand{\N}{\mathbb{N}}
\newcommand{\Z}{\mathbb{Z}}

\def \bdry {\partial}
\newcommand{\K}{{\rm K3}}

\def \x {\times}
\def \eu{{\text{e}}}

\def \h  {{\textsf{L}}}
\def \hT {\widetilde{{\textsf{{L}}}}}

\begin{document}

\title[Positive factorizations of mapping classes]
{Positive factorizations of mapping classes}

\author[R. \.{I}. Baykur]{R. \.{I}nan\c{c} Baykur}
\address{Department of Mathematics and Statistics, University of Massachusetts, Amherst, MA 01003-9305, USA}
\email{baykur@math.umass.edu}

\author[N. Monden]{Naoyuki Monden}
\address{Department of Engineering Science, Osaka Electro-Communication University, Hatsu-cho 18-8, Neyagawa, 572-8530, Japan}
\email{monden@isc.osakac.ac.jp}

\author[J. Van Horn-Morris]{Jeremy Van Horn-Morris}
\address{Department of Mathematical Sciences, The University of Arkansas, \newline 
\indent Fayetteville, AR 72701,  USA }
\email{jvhm@uark.edu }


\begin{abstract}
In this article, we study the maximal length of positive Dehn twist factorizations of surface mapping classes. In connection to fundamental questions regarding the uniform topology of symplectic $4$-manifolds and Stein fillings of contact $3$-manifolds coming from the topology of supporting Lefschetz pencils and open books, we completely determine which boundary multitwists admit arbitrarily long positive Dehn twist factorizations along nonseparating curves, and which mapping class groups contain elements admitting such factorizations. Moreover, for every pair of positive integers $g,n$, we tell whether or not there exist genus-$g$ Lefschetz pencils with $n$ base points, and if there are, what the maximal Euler characteristic is whenever it is bounded above. We observe that only symplectic $4$-manifolds of general type can attain arbitrarily large topology regardless of the genus and the number of base points of Lefschetz pencils on them.
\end{abstract}

\maketitle

\setcounter{secnumdepth}{2}
\setcounter{section}{0}

\section{Introduction} 

Let $\Sigma_g^n$ be a compact orientable genus $g$ surface with $n$ boundary components, and $\Gamma_g^n$ denote the mapping class group composed of orientation-preserving homeomorphisms of $\Sigma_g^n$ which restrict to identity along $\partial \Sigma_g^n$, modulo isotopies fixing the same data. We denote by $t_c \in \Gamma_g^n$, the positive (right-handed) Dehn twist along the simple closed curve $c \subset \Sigma_g^n$. If $\Phi= t_{c_l} \cdots t_{c_1}$ in $\Gamma_g^n$, where $c_i$ are nonseparating curves on $\Sigma_g^n$, we call the product of Dehn twists a \emph{positive factorization} of $\Phi$ in $\Gamma_g^n$ of \emph{length $l$}. 

Our motivation to study positive factorizations come from their significance in the study of Stein fillings of contact $3$-manifolds \cite{Giroux, LP} and that of symplectic $4$-manifolds \cite{D} via Lefschetz fibrations and pencils. Provided all the twists are along homologically essential curves, a positive factorization of a mapping class in $\Gamma_g^n$, $n  \geq 1$, prescribes an \textit{allowable Lefschetz fibration} which supports a Stein structure filling the contact structure compatible with the induced genus $g$ open book on the boundary with $n$ binding components \cite{LP, AO}. Similarly, a positive factorization of a boundary multitwist ${\Delta}$, i.e. the mapping class $t_{\delta_1} \cdots t_{\delta_n}$ for $\delta_i$ boundary parallel curves, describes a genus $g$ \textit{Lefschetz pencil} with $n$ base points. Conversely, given any allowable Lefschetz fibration or a Lefschetz pencil, one obtains such a factorization.

\enlargethispage{0.2in}
The main questions we will take on in this article are the following: 

\noindent \textbf{Q1} \textit{When does the page $F \cong \Sigma_g^n$ of an open book impose an a priori bound on the Euler characteristics of allowable Lefschetz fibrations with regular fiber $F$ filling it?} 

\vspace{0.1in}
\noindent \textbf{Q2} \textit{When does the fiber $F \cong \Sigma_g$ and a positive integer $n$ imply an a priori bound on the Euler characteristics of allowable Lefschetz pencils with regular fiber $F$ and $n$ base points? When it does, what is the largest possible Euler characteristic? More specifically,  for which $g,n$ do there exist genus-$g$ Lefschetz pencils with $n$ base points?} 

By Giroux, contact structures on $3$-manifolds up to isotopies are in one-to-one correspondence with supporting open books up to positive stabilizations \cite{Giroux}. Moreover, a contact $3$-manifold $(Y, \xi)$ admits a Stein filling $(X, J)$ if and only if $(Y, \xi)$ admits a \textit{positive open book}, i.e. an open book whose monodromy can be factorized into positive Dehn twists along homologically essential curves on the page \cite{LP}. Since $b_1(X) \leq b_1(Y)$ for any Stein filling $X$, the Question~$1$ above, \textit{up to stabilizations}, amounts to asking when the page of an open book (i.e. its genus $g$ and number of binding components $n$) on a contact $3$-manifold implies a uniform bound on the topology of its Stein fillings. On the other hand, the Question~$2$ can be seen as a special version of Question~$1$ with page $\Sigma_g^n$ and the particular monodromy $\Delta$.
  
Related to our focus is the following natural function defined on mapping class groups: The \emph{positive factorization length}, or \emph{length} (as \emph{the} length, in this paper) of a mapping class $\Phi$, which we denote by $\h(\Phi)$, is defined to be the supremum taken over the lengths of all possible positive factorizations of $\Phi$ along nonseparating curves, and it is $-\infty$ if $\Phi$ does not admit any positive factorization. $\h$ is a superadditive function on $\Gamma_g^n$ taking values in $\N \, \cup \{\pm \infty \}$. It is easy to see that $\h < + \infty$ translates to having the uniform bound in the above questions. 

We will investigate the length of various mapping classes, leading to a surprisingly diverse picture, where our results, will in particular provide complete solutions to the above problems. Below, $\delta_i$ always denote boundary parallel curves along distinct boundary components of $\Sigma_g^n$ for $i=1, \ldots, n$. 

\smallskip
\begin{mthma} 
Let $\Delta = t_{\delta_1} \cdots t_{\delta_n}$ on $\Sigma_g^{n}$, where $n  \geq  1$
\footnote{As we will review while proving the above theorem, it is well-known that for $\Delta = t_{\delta_1} \cdots t_{\delta_n}$ on $\Sigma_g^{n'}$ with $n < n'$, we have $\h(\Delta)= -\infty$. We have therefore expressed our results only for the nontrivial case $n=n'$.} 
Then 
\begin{equation*}
\h(\Delta) = \left\{\begin{array}{rl}   -\infty & \mbox{       \,\,\,      if } g=1 \mbox{ and } n> 9 \mbox{, or } g \geq 2 \mbox{ and } n> 4g+4 \\
+\infty & \mbox{       \,\,\,       if } n \leq 2g-4 \\ 
\mbox{finite} & \mbox{   \,\,\,          otherwise.} \end{array}\right.
\end{equation*}
\emph{When finite}, the exact value of $\h(\Delta)$ is 
\begin{equation*}
\h(\Delta) = \left\{\begin{array}{rl}  12 & \mbox{     if } g=1 \\
40 & \mbox{     if } g=2 \\ 
6g+18 & \mbox{    if }3 \leq g \leq 6 \, \, \, \, \, \, \, \, \, \, \, \, \, \, \, \, \, \, \, \, \, \, \, \, \, \, \, \, \, \, \, \, \, \, \, \, \, \, \, \, \, \, \, \, \, \, \, \, \, \, \, \, \, \, \, \, \, \, \, \, \, \, \, \, \, \, \, \, \, \, \, \, \, \, \, \, \, \, \,   \, \, \, \, \, \\ 
8g+4 & \mbox{    if }g \geq 7. \end{array}\right.
\end{equation*}
In particular, when $\h(\Delta)$ is finite, its value depends solely on $g$, and not $n$.
\end{mthma}

\smallskip
Translating this to Lefschetz pencils; for every pair of fixed positive integers $g, n$, Theorem~A allows us to tell (i) if there are any symplectic $4$-manifolds admitting a genus-$g$ Lefschetz pencil with $n$ base points (and with only irreducible singular fibers), (ii) when the Euler characteristic of these $4$-manifolds can get arbitrarily large, and (iii) if bounded, what the largest Euler characteristic exactly is. Parts (i) and (ii) completely answer the Question~5.1 in \cite{DKP}. In the course of the proof, we will note that only symplectic $4$-manifolds of general type, i.e. of Kodaira dimension $2$, realize arbitrarily large Euler characteristic. In contrast, when there is a uniform bound, we will see that the largest Euler characteristic can be realized by a symplectic $4$-manifold of Kodaira dimension $-\infty, 0$ or $2$. 

Our second theorem, inspired by partial observations in \cite{DKP}, shows that when the Question~1 is formulated for higher powers of the boundary multitwist $\Delta$ (i.e. for Lefschetz fibrations with sections of self-intersection less than $-1$ instead of pencils) or for the product of $\Delta$ with a single Dehn twist along a nonseparating curve, the uniform bounds in Theorem~A disappear for all $g \geq 2$.

\smallskip
\begin{mthmb} 
Let $\Delta= t_{\delta_1} \cdots t_{\delta_n}$ on $\Sigma_g^n$, where $n \geq 1$ any integer, and $a$ be any nonseparating curve on $\Sigma$. Then 
\begin{enumerate}
\item $\h(\Delta^k)=12k$ \ \, \, if $g=1$, $k \geq 1$, and $n \leq 9$ \
\item $\h(\Delta^k)=+\infty$ \, \, if $g \geq 2$, $k \geq 2$, \
\item $\h(\Delta \, t_a)= + \infty$ \, if $g \geq 2$.
\end{enumerate}
\end{mthmb}

A simple variation of $\h$ is obtained by allowing separating Dehn twists along \textit{homologically essential} curves on $\Sigma_g^n$ in the factorizations, which we denote by $\hT$. Clearly $\h(\Phi) \leq \hT(\Phi)$ for all $\Phi$, and $\h(\Gamma_g^n) \subset \hT(\Gamma_g^n)$. Our last theorem determines the \textit{full image} of $\h$ and $\hT$ on mapping class groups $\Gamma_g^n$:

\begin{mthmc} 
The image of $\Gamma_g^n$, for $n \geq 1$, under $\hT$ and $\h$ is  
\begin{enumerate}
\item $\hT(\Gamma_g^n)= \h(\Gamma_g^n) = \N \cup \{-\infty\}$ \, if $g=0$ and $n \geq 2$, or if $g=1$, \
\item $\hT(\Gamma_g^n)= \h(\Gamma_g^n) = \N \cup \{\pm \infty\}$ \, if $g \geq 2$.
\end{enumerate}
\end{mthmc}

Theorem~C, along with parts of Theorems~A and ~B, records the existence of mapping classes with arbitrarily long positive factorizations. In the course of its proof we will in fact spell out mapping classes (as multitwists along nonseparating curves) with positive factorizations of \textit{unique} lengths. 

Any $\Phi$ with $\h(\Phi)= +\infty$ provides an example of a contact $3$-manifold with arbitrarily large Stein fillings. The mapping class $\Phi$ prescribes an open book, which in turn determines a contact $3$-manifold by Thurston-Winkelnkemper, whose Stein fillings are obtained from the allowable Lefschetz fibrations corresponding to respective positive factorizations. We therefore extend our earlier results in \cite{BV1, BV2}, and obtain many more counter-examples to Stipsicz's conjecture \cite{Stipsicz2}, which predicted an a priori bound on the Euler characteristics of Stein fillings. Clearly, any contact $3$-manifold Stein cobordant to one of these examples also bear the same property. Since having a supporting open book with infinite length monodromy is a contact invariant, our detailed analysis summarized in the results above can be used to distinguish contact structures on $3$-manifolds.

\vspace{0.2in}
The novelty in the proofs of the above theorems is the engagement of essentially four different methods: 

\noindent { \textbf{1.} \textit{Underlying symplectic geometry and Seiberg-Witten theory.}} The bounds and calculations of the maximal length in the finite cases in Theorem~A will follow from our analysis of the underlying Kodaira dimension of the symplectic Lefschetz pencils corresponding to these factorizations. Here the symplectic Kodaira dimension will provide a useful way to organize the essential information. 
Indeed, we will observe that the only classes of $4$-manifolds yielding pencils or fibrations with unbounded Euler characteristic are symplectic $4$-manifolds of general type, i.e. of Kodaira dimension $2$. Both our Kodaira dimension calculations \cite{Sato, BH} (even the very fact that the Kodaira dimension is a well-defined invariant \cite{Li1}) and the sharp inequalities we obtain heavily depend on Taubes' seminal work on pseudoholomorphic curves and Seiberg-Witten equations on symplectic $4$-manifolds. 

\noindent { \textbf{2.} \textit{Dehn monoid and right-veering.}} Realizing the finite lengths in Theorem~B (Proposition~\ref{unique}), as well as our recap of previous results which establishes the lack of any positive factorizations for certain mapping classes in Theorems~A and C (Proposition~\ref{nonpositive}) will be obtained using Thurston type right-veering arguments \cite{ShortWiest} \cite{HKM}, and will rely on the structure of the positive Dehn twist monoid of $\Gamma_g^n$. 

\noindent { \textbf{3.} \textit{Homology of mapping class groups of small genera surfaces.}} The precise calculation for the genus $1$ case in Theorem~B, and the bounds in Theorem~C for low genera cases (Proposition~\ref{smallgenera}), will follow from our complete understanding of the first homology group of the corresponding mapping class groups. 

\noindent { \textbf{4.} \textit{Constructions of new monodromy factorizations.}} We will show that an artful application of braid, chain and lantern substitutions applied to carefully tailored mapping class group factorizations allows one to obtain arbitrarily long positive factorizations in Theorems~A and~B (Theorems~\ref{thm1},~\ref{thm2} and~\ref{thm3}). This greatly extends the earlier array of partial results of \cite{BKM, BV1, BV2, DKP} to the possible limits of these constructions as dictated by our results above.

\smallskip
The outline of our paper is as follows: In Section~\ref{Sec:finite}, we discuss mapping classes with bounded lengths, and in Section~\ref{Sec:infinite} we construct those with infinite lengths. These results will be assembled to complete the proofs of our main theorems in Section~\ref{Sec:final}, where we will also present a couple more results on lengths of mapping classes and special subgroups of mapping class groups, and discuss some related questions.

\vspace{0.1in}
\noindent \textit{Acknowledgements.} The first author was partially supported by the NSF Grant 
DMS-$1510395$ and the Simons Foundation Grant $317732$ . The second author was supported by the Grant-in-Aid for Young Scientists (B) (No. 13276356), Japan Society for the Promotion of Science. The third author was partially supported by the Simons Foundation Grant $279342$.

\vspace{0.1in}
\section{Mapping classes with finite lengths} \label{Sec:finite}

Here we investigate various examples of mapping classes with no positive factorizations or with only factorizations of bounded length. We first probe mapping classes on small genera surfaces, as well as those with small compact support in the interior, who have unique lengths. We then move on to showing that boundary multitwists involving too many boundary components have an a priori bound on the length of their positive factorizations.

\subsection{Simple mapping classes with prescribed lengths} \

There are two tools that we will use to bound the number of Dehn twists in a given factorization. The first uses the fact that the mapping class group $\Gamma_g^n$ for $g=0,1$ surjects to $\Z$, based on $H_1(\Gamma_g^n; \Z)$ having a $\Z$ component. Secondly, we will use \emph{right-veering} methods of Thurston \cite{ShortWiest} and Honda-Kazez-Mati\'c \cite{HKM}.

We begin with analyzing the $g=0,1$ case, for any $n \geq 1$. When $g=0$, equivalent arguments were given by Kaloti in \cite{Kaloti} and Plamenevskaya in \cite{P}.

\begin{proposition} \label{smallgenera}
For $g \leq 1$ any positive factorization of $\Phi \in \Gamma_g^n$ along homologically essential curves has bounded length. So for any $\Phi$, $\hT(\Phi)$ is bounded above. Further, when $g=1$, the length of any factorization into non-separating Dehn twists is fixed.
\end{proposition}

\begin{proof}
While a careful look at $H_1(\Gamma_g^n)$ will yield more information, a bound on the number of non-trivial Dehn twists in a positive factorization of a mapping class $\Phi$ can be obtained by capping off to the bases cases of $g=0, n=2$ or $g = 1, n=1$. 

\noindent {\it Genus 0:} For genus 0, the base case is $n=2$. The mapping class group of the annulus $\Gamma_0^2$ is isomorphic to $\Z$ and generated by the right-handed Dehn twist about the core of the annulus. In this case, $\hT (\Phi) = [\Phi]$ and moreover, a positive factorization, if it exists, is unique. 

When $n>2$, one can fix an outer boundary component, and identify $\Sigma$ with a disk with holes. The homomorphism induced by capping off all but a single interior boundary component $\bdry_i$ counts the number of Dehn twists (in any factorization) that enclose $\bdry_i$. Every essential curve must enclose at least one interior boundary component and so shows up in at least one of these counts. Adding up the images of $\Phi$ for all of these homomorphisms then gives a bound on the number of Dehn twists in any positive factorization of $\Phi$.

\noindent {\it Genus 1:} For genus 1, the base case is $n=1$. The first homology of the mapping class group $H_1(\Gamma_1^1)$ isomorphic to $\Z$ and generated by the right-handed Dehn twist about any nonseparating curve. The boundary Dehn twist has $[t_\bdry]=12$. In this case, $\h (\Phi) = [\Phi]$ and $\hT \leq [\Phi]$. 

When $n>1$, then just as above, one can cap off all but one boundary component and calculate the value of ${\Phi}$ there. Any essential curve will remain essential for at least one of these maps, and so adding up the values of all of the images of $\Phi$ will give a bound on the length of any factorization of $\Phi$ into Dehn twists along homologically essential curves and an upper bound on $\hT(\Phi)$

Even better, though, any non-separating curve will remain non-separating after every capping, so the length of any positive factorization into non-separating Dehn twists can be found by capping off all but one boundary component of $\Sigma$ and calculating $[{\Phi}]$ there. This determines $\h(\Phi)$ on the nose.
\end{proof}

\begin{remark}
Notice that combining the above with a theorem of Wendl \cite{Wendl} we recover the following theorem of \cite{P, Kaloti}: \textit{If the open book prescribed by $\Phi \in \Gamma_g^n$, $n \geq 1$ is stably equivalent to a planar open book, then $\hT(\Phi)$ is finite.} Equivalently, a mapping class $\Phi \in \Gamma_g^n$, $n \geq 1$ with $\hT = + \infty$ cannot be stably equivalent to a mapping class $\Psi \in \Gamma_0^m$, for any $m \geq 1$. 
\end{remark}

We now move on to producing particular mapping classes with prescribed finite lengths in $\Gamma_g^n$ for any $g,n \geq 1$, for which we first review the notion of right-veering \cite{HKM}. Let $\alpha$ and $\beta$ two properly embedded oriented arcs in an oriented surface $\Sigma$ with $\partial \Sigma \neq 0$, having the same endpoints $\alpha(0) = \beta(0) = x_0$ and $\alpha(1) = \beta(1)$ on $\partial \Sigma$. Choose a lifted base point $\tilde{x}_0$ of $x_0$ and lifts to the universal cover $\tilde{\alpha}$ and $\tilde{\beta}$ of the arcs $\alpha$ and $\beta$ starting at $\tilde{x}_0$. We say $\beta$ is \emph{to the right of $\alpha$ at $x_0$} if the boundary component of $\tilde{\Sigma}$ containing $\tilde{\beta}(1)$ is to the right of that containing $\tilde{\alpha}(1)$ when viewed from $\tilde{x}_0$. 

When $\Sigma$ is an annulus, right veering is equivalent to being a positive power of the right-handed Dehn twist. For surfaces {other than the annulus}, one thinks of $\tilde{\Sigma}$ as sitting in the Poincar\'e disk model of the hyperbolic plane, lift the arcs to geodesics, and consider the endpoints $\tilde{\alpha}(1)$ and $\tilde{\beta}(1)$ radially from the boundary of $\tilde{\Sigma}$ to the boundary of the disk. If the oriented path from $\tilde{\beta}(1)$ to $\tilde{\alpha}(1)$, avoiding $\tilde{x}_0$, has the same orientation as the boundary orientation induced by the disk, then $\beta$ is to the right of $\alpha$. (More generally, one can take homotopic representatives of $\alpha$ and $\beta$ which intersect minimally and are transverse at $x_0$, and ask whether in the neighborhood of $x_0$, $\beta$ lies on the right or left of $\alpha$.) A diffeomorphism $\Phi$ is called \emph{right-veering} if for every base point $x_0$ and every properly embedded arc $\alpha$ starting at $x_0$, either $\Phi(\alpha)$ is \emph{to the right} of $\alpha$ at $x_0$ or $\Phi(\alpha)$ is homotopic to $\alpha$, fixing the end points. It is straightforward to see that this property is well-defined for an isotopy class of $\Phi$ rel boundary, and is independent of the choice of the base point. Hence we call a mapping class $\Phi$ to be right-veering if any representative of it is. It turns out that $Veer_g^n$, generated by right-veering mapping classes in $\Gamma_g^n$ is a monoid of $\Gamma_g^n$, and contains $Dehn_g^n$, generated by all positive Dehn twists in $\Gamma_g^n$ as a submonoid \cite{HKM}.

\begin{prop} \label{prop:dehn pos} 
Let $\prod_{i=1}^{n} t_{c_i}$ be a positive Dehn twist factorization of a mapping class element $\Phi$. If $\Phi$ preserves an arc $\alpha$, then every curve $c_i$ is disjoint from the arc $\alpha$. 
\end{prop}

\begin{proof}
If $\Phi$ is a mapping class which preserves the homotopy class of an arc $\alpha$, then for any factorization of $\Phi$ as a product of right-veering maps $\Phi = \Phi_r  \circ \ldots \circ \Phi_1$, we can see that each of $\Phi_i$  also preserve $\alpha$. For if any $\Phi_i$ moves $\alpha$, then it has to send it to the right, after which every $\Phi_j$, $i < j \leq r$ either fixes it, or sends it further to the right. Since Dehn twists are right-veering, the proposition follows.
\end{proof}

A \emph{multicurve} $C$ on $\Sigma$, which is a collection of disjoint simple closed curves, is said to be \emph{nonisolating} if every connected component of $\Sigma \setminus C$ contains a boundary component of $\Sigma$. Next, we observe that Dehn twist along such $C$ can realize any finite length $l$. 

\begin{prop} \label{unique}
Let $C = c_1 \cup \cdots \cup c_r$ be a {nonisolating} multicurve on $\Sigma$ and $m_1,\dots, m_r$ non-negative integers. Then the multitwist $\prod_{i=1}^r t_{c_i}^{m_i}$ has a {unique factorization} in $\Gamma_g^n$ into positive Dehn twists and hence $\hT\left(\prod_{i=1}^r t_{c_i}^{m_i}\right) =  m_1 + \cdots + m_r$. 
\end{prop}

\begin{proof} 
Let $\Phi = \prod_{i=1}^r t_{c_i}^{m_i}$ with $m_1,\dots, m_r$ non-negative integers as above. By Proposition~\ref{prop:dehn pos}, for any positive factorization of $\Phi$, each factor fixes every arc which is disjoint from $C$. Since $C$ is nonseparating, we can find arcs disjoint from $C$ which cut $\Sigma$ into disjoint annuli that respectively deformation retract onto $\bigcup c_i$. Thus all factors are supported on these annuli, must be Dehn twists along these annuli, and give the obvious factorization $\prod_{i=1}^n t_{c_i}^{m_i}$. 
\end{proof}

Lastly, we note a general source of mapping classes in $\Gamma_g^n$, $n \geq 1$, with no positive factorizations:

\begin{proposition}\label{nonpositive}
If $\Phi$ is a non-trivial element in $Veer_g^n$, $n \geq 1$, then $\Phi^{-1}$ is not. Thus if $\Phi$ admits a nontrivial positive factorization, then $\Phi^{-1}$ is not right-veering. In particular, $\h(\Delta^k) = \hT(\Delta^k) =- \infty$ for $k <0$ and $\h(1) = \hT(1) = 0$. 
\end{proposition}

\begin{proof} If $\Phi$ is non-trivial and right-veering, then it moves at least one arc $\alpha$ to the right. Then $\Phi^{-1}$ sends $\Phi(\alpha)$ to $\alpha$; that is, to the left.
\end{proof}

\noindent The particular case we noted above, that $\Delta^k$ for $k \leq 0$ does not admit any nontrivial positive factorizations, was first observed by Smith \cite{Smith} (whose arguments are similar to ours) and also by Stipsicz \cite{Stipsicz} (who used Seiberg-Witten theory).

\vspace{0.2in}
\subsection{Boundary multitwists with finite lengths} \

We are going to prove:

\begin{thm} \label{Kodaira}
Let $\Delta = t_{\delta_1} \cdot t_{\delta_2} \ldots \cdot t_{\delta_n}$ be the boundary multitwist on $\Sigma_g^n$, with $n  \geq 2g-3 \geq 0$. If $n > 4g+4$, then $\h(\Delta) =-\infty$. If $n \leq 4g+4$, we have
\begin{equation*}
\h(\Delta) = \left\{\begin{array}{rl}  
40 & \mbox{     if } g=2 ,   \\ 
6g+18 & \mbox{    if }3 \leq g \leq 6,  \, \, \, \, \, \, \, \, \, \, \, \, \, \, \, \, \, \, \, \, \, \, \, \, \, \, \, \, \, \, \, \, \, \, \, \, \, \, \, \, \, \, \, \, \, \, \, \, \, \, \, \, \, \, \, \, \, \, \, \, \, \, \, \, \, \, \, \, \, \, \, \, \, \, \, \, \, \, \,   \, \, \, \, \, \\ 
8g+4 & \mbox{    if }g \geq 7. \end{array}\right.
\end{equation*}
In particular, when $\h(\Delta)$ is finite, its value depends solely on $g$, and not $n$.
\end{thm}

Let us briefly review here the notion of \emph{symplectic Kodaira dimension} we will repeatedly refer to in our proof of this theorem. The reader can turn to \cite{Li1} for more details. First, we recall that a symplectic $4$-manifold $(X, \omega)$ is called \emph{minimal} if it does not contain any embedded symplectic sphere of square $-1$, and also that it can always be blown-down to a minimal symplectic $4$-manifold $(X_{\text{min}}, \omega')$. Let $\kappa_{X_{\text{min}}}$ be the canonical class of a minimal model $(X_{\text{min}}, \omega_{\text{min}})$. We define the symplectic Kodaira dimension of $(X, \omega)$, denoted by $\kappa=\kappa(X,\omega)$ as follows:
\[
\kappa(X,\omega)= \left\{\begin{array}{rl}  -\infty& \mbox{if
}\kappa_{X_{\text{min}}}\cdot[\omega_{\text{min}}]<0 \mbox{ or } \kappa_{X_{\text{min}}}^{2}<0 \\
0 & \mbox{if } \kappa_{X_{\text{min}}}\cdot[\omega_{\text{min}}]=\kappa_{X_{\text{min}}}^{2}=0\\ 1 &
\mbox{if }\kappa_{X_{\text{min}}}\cdot[\omega_{\text{min}}]>0\mbox{ and
}\kappa_{X_{\text{min}}}^{2}=0\\2& \mbox{if }\kappa_{X_{\text{min}}}\cdot[\omega_{\text{min}}]>0\mbox{
and }\kappa_{X_{\text{min}}}^{2}>0\end{array}\right.
\]
Here $\kappa$ is independent of the minimal model $(X_{\text{min}}, \omega_{\text{min}})$ and is a smooth invariant of the $4$-manifold $X$.

\begin{proof}[Proof of Theorem~\ref{Kodaira}]

Assume that $\Delta$ admits a positive Dehn twist factorization $W= t_{c_l}  \cdots t_{c_1}$ along nonseparating curves $c_i$ in $\Gamma_g^n$. Let $(X,f)$ be the genus-$g$ Lefschetz fibration with $n$ disjoint $(-1)$-sphere sections, $S_1, \ldots, S_n$, associated to this factorization. We can support $(X,f)$ with a symplectic form $\omega$, with respect to which all $S_j$ are symplectic as well. Note that by the hypothesis, $g \geq 2$ and $n \geq 1$, the latter implying that $X$ is \textit{not} minimal. 

For fixed $g$, $l$ is maximized if and only if the Euler characteristic of $X$ is, where $\eu(X)=4-4g+l$. Whereas fixing $n$, along with $g$, will play a role in narrowing down the possible values of the symplectic Kodaira dimension $\kappa(X)$. We will read off $\kappa(X)$ based on the number of $(-1)$-sphere sections of $f$. In principal, we need to know that there are no other disjoint $(-1)$-sphere sections than $S_1, \ldots, S_n$, that is, there are no lifts of the positive factorization to a boundary multitwist in $\Gamma_g^{n'}$ with $n'>n$. We will overcome this issue by simply presenting our arguments starting with $\kappa= - \infty$ and going up to non-negative $\kappa=0$ cases. (Meanwhile, it will become evident that $\kappa=1$ and $2$ cases cannot occur, so the proof will boil down to realizing and comparing the bounds we obtain in $\kappa=0$ and $1$ cases.)

If $n > 2g-2$, we can blow-down the $n$ $(-1)$-sphere sections $S_1, \ldots, S_n$  to derive a symplectic surface $F'$ from a regular fiber $F$ of $f$, which has genus $g$ and self-intersection $n$. Since the Seiberg-Witten adjunction inequality  
\[ 2g-2 = -\eu(F') \geq [F']^2 + | \beta \cdot F'| \geq [F']^2 = n \, , \]
is violated by $F'$, we conclude that $X'$ (and thus $X$) should be a rational or ruled surface \cite{LiLiu}. These are precisely the symplectic $4$-manifolds with Kodaira dimension $\kappa = - \infty$.

If $n = 2g-2$, and $X$ is not rational or ruled, it follows from Sato's work on the canonical class of genus $g \geq 2$ Lefschetz fibrations on non-minimal symplectic $4$-manifolds that the canonical class $K_X$ can be represented by the sum $\sum _{j=1}^n S_j$ of the exceptional sphere sections \textit{in $H_2(X ; \Q)$} \, (see \cite{Sato} and also \cite{BH}).\footnote{As the author's argument in \cite{Sato} is based on positive intersections of holomorphic curves, it essentially captures the homology class of $K_X$ only in $\Q$ coefficients -- which otherwise would lead to a contradiction for pencils on the Enriques surface. Hence we quote here the result with this small correction \cite{BH}.} Blowing-down all $S_j$ we get $K_{X'}=0$ in $H_2(X; \Q)$. In particular, the canonical class is torsion, and so $X$ is a blow-up of a symplectic Calabi-Yau surface, $\kappa(X)=0$. The minimal model of $X$ should then have the rational homology type of a torus bundle over a torus, the Enrique surface, or the $\K$ surface by the work of Li and independently of Bauer \cite{Bauer, Li2, Li3}.

Now if $ n = 2g -3$, and  $X$ is not rational or ruled or a (blow-up of a) symplectic Calabi Yau surface, then the collection $\sum_{j=1}^n S_j$ realizes the maximal disjoint collection of representatives of its exceptional classes intersecting the fiber. It therefore follows from Sato's work in \cite{Sato} that, \textit{provided $g \geq 3$} for the genus-$g$ Lefschetz fibration on $X$, the canonical class of $X$ is represented by $2 S_1 + \sum_{j=2}^{n} S_j + R$, where $S_1$ is a distinguished $S_j$ we get by relabeling if necessary, and more importantly, $R$ is a genus-$1$ irreducible component of a reducible fiber with $[R]^2=-1$. The latter condition however is not realized by any Lefschetz fibration with only nonseparating vanishing cycles, which allows us to rule out this case. Finally, in the remaining $g=2$ and $n=1$ case, it follows from Smith's analysis of genus-$2$ pencils in \cite{Smith1}[Theorem~5.5] that the maximal number of irreducible singular fibers is $l= 40$. 

We have thus seen that for $ n \geq 2g-3$, the $\kappa=2$ and $\kappa=1$ cases are already ruled out. It therefore suffices to discuss the $\kappa = -\infty$ and $0$ cases, and compare the largest $l$ we get in these cases to determine the winner, all while remembering we have an additional candidate in the $g=2$ case as noted above.

Let $\kappa(X)= - \infty$. Because $X$ is not minimal, and because $\CP \# \CPb$ does not admit any genus $g > 0$ Lefschetz fibration with a $(-1)$-section\footnote{If it did, one would get a homology class $F=aH$ with $F^2=1$, where $H$ is the generator of  $H_2(\CP)$. This is only possible when $a=1$, which implies that the fiber genus is zero.}, we have $X \cong S^2 \x \Sigma_h \#m \CPb$ for some $h \geq 0, m \geq 1$.  We have $\eu(X)=2(2-2h)+m=4-4g+l$, so 
\begin{equation}\label{first}
l = 4(g -h) + m \leq 4g + m.
\end{equation}
On the other hand, it was shown in \cite{Stipsicz1} that $4(b_1(X) - g) + b^-_2 (X) \leq 5 b^+_2 (X)$. For $X$ with $b_1(X) = 2h$, $b_2^+ (X) = 1$, and $b_2^- (X) = m+1$, we get 
\begin{equation}\label{second}
m \leq 4-4(2h-g) \leq  4+4g.
\end{equation}
Combining the inequalities (\ref{first}) and (\ref{second}), we conclude that  $l \leq 8g+4$. 

The first conclusion, namely that $h(\Delta)= -\infty$ when $n>4g+4$, is rather immediate. Here $X \cong S^2 \x \Sigma_h \#m \CPb$, and either $h>0$ and we have $m \geq n$ or $h=0$ and we only have $m+1 \geq n$. The inequality (\ref{second}) above, combined with our assumption $n>4g+4$, implies that the former is impossible, wheras the latter can hold only if $m=4g+5$. However, we claim that there is no genus-$g$ Lefschetz pencil on $\CP$ with $m=4g+5$ base points. Let $H$ represent the generator of $H_2(\CP)$, and $F= aH$ represent the potential fiber class of the pencil. Since there exists a symplectic form for which the fiber is symplectic, and since $\CP$ has a unique symplectic structure up to deformations and symplectomorphisms, we can invoke the adjunction equality as
\[2g-2=F^2 - 3H \cdot F = 4g+5 -3a \, , \]
so $3a=2g+7$. As $F$ is a fiber class, $a^2= F^2 = 4g+5$ should be satisfied as well. The only possible solution is when $a=3$, which is the case of a $g=1$ pencil (indeed, the well-known case of an elliptic pencil on $\CP$) we have excluded from our discussion (recall our hypothesis  $n \geq 2g-3 \geq 0$).

Moreover, note that the equality $l=8g+4$ holds only if $m = 4g + 4$ and $h=0$. There exist such genus-$g$ Lefschetz fibrations with $l=8g+4$ irreducible fibers and $m=4g+4$ sections of square $-1$ on $\CP \# (4g+5) \CPb$; see \cite{Tanaka, SaSa}. 

Now, let $\kappa(X)=0$. Recall that $l$ is maximal when $\eu(X)$ is. Rational cohomology $\K$ has the largest Euler characteristics among all minimal candidates, and as discussed above, one can hope to have a genus-$g$ Lefschetz fibration on at most $2g-2$ blow-ups of a symplectic Calabi-Yau surface. It follows that the maximal $l$ is realized when $X$ is a rational cohomology $\K$ surface blown up $2g-2$ times. So $2(2-2g)+l = 24+2g-2$, implying \, $l=6g+18$. Such Lefschetz fibrations on symplectic Calabi Yau $\K$ surfaces are constructed in \cite{BM}; also see \cite{Smith1}[Proof of Theorem~3.10]. 

Hence all remaining conclusions of the theorem follow from a comparison of the maximal $l$ we get in the $\kappa= - \infty$ and $\kappa =0$ cases, along with the additional ($\kappa=2$) case when $g=2$.
\end{proof}

\begin{remark} 
As seen in our proof, there is an a priori upper bound, determined by the genus $g$ and the number of base points $n$, on the number of critical points of Lefschetz pencils when $\kappa < 1$, and for pencils with only irreducible fibers when $\kappa \leq 1$. So arbitrarily large topology is specific to pencils on \textit{symplectic $4$-manifolds of general type}, i.e. when $\kappa =2$. In contrast, when the uniform topology is bounded, the maximal Euler characteristic for a genus-$g$ Lefschetz pencil with $n$ base points can be realized by an $(X,f)$ with $\kappa =2$ when $g=2$, $\kappa = 0$ when $3 \leq g \leq 7$, and $\kappa = - \infty$ when $g \geq 7$. 
\end{remark}

\vspace{0.1in}
\section{Mapping classes with infinite lengths} \label{Sec:infinite}

Here we will construct arbitrarily long positive factorizations of various mapping classes involving boundary multitwists in $\Gamma_g^n$, for $g \geq 2$, $n \geq 1$.

\subsection{Preliminary results} \

We begin with a brief exposition of various recent results on arbitrarily long positive factorizations in \cite{BKM, BV1, BV2, DKP}, which creates leverage to many of our results to follow. We hope that the proofs given below will help with making the current article self-contained in this aspect. 

First examples or arbitrarily long positive factorizations were produced in \cite{BKM} by Korkmaz, and the first two authors of this article, for a \textit{varying family} of single commutators in $\Gamma_g^2$, for any $g \geq 2$. The proof of this result is based on the following well-known relations: Let $c_1, c_2,\ldots,c_{2h+1}$ be simple closed curves on $\Sigma_g^n$ such that $c_i$ and $c_j$ are disjoint if $|i-j|\geq 2$ and that $c_i$ and $c_{i+1}$ intersect at one point. Then, a regular neighborhood of $c_1\cup c_2 \cup\cdots \cup c_{2h+1}$ is a subsurface of genus $h$ with two boundary components, $b_1$ and $b_2$. 
We then have the \textit{chain relations}:
\begin{align*}
&t_{b_1}t_{b_2} = (t_{c_1}t_{c_2}\cdots t_{c_{2h+1}})^{2h+2} = (t_{c_{2h+1}}\cdots t_{c_2}t_{c_1})^{2h+2}. 
\end{align*}
\begin{figure}[ht]
\centering
\includegraphics[scale=.75]{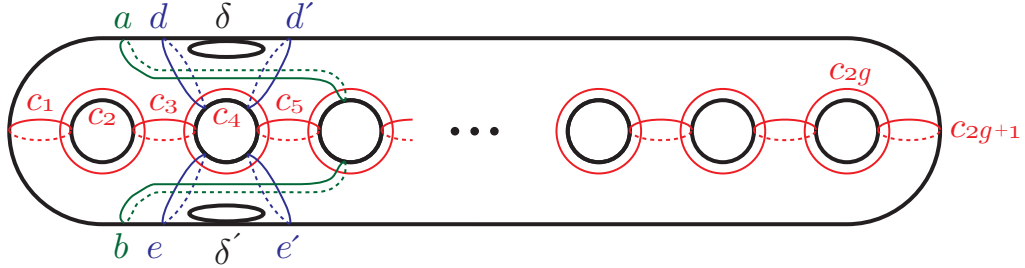}
\caption{The curves $c_1,c_2,\ldots,c_{2g+1}$ and $a,b,d,d^\prime,e,e^\prime$ and the boundary curves $\delta,\delta^\prime$ on $\Sigma_g^2$.}
\label{curves1}
\end{figure}
\noindent Now for a chain of length $5$, we get $t_d t_e = (t_{c_1}t_{c_2}t_{c_3})^4$ and by applying the below relation (\ref{hurwitz2}) to $(t_{c_1}t_{c_2}t_{c_3})^2$, we obtain 
\begin{align*}
t_d t_e = (t_{c_1}t_{c_2}t_{c_3})^2 t_{c_2}t_{c_3}t_{c_1}t_{c_2}t_{c_3}t_{c_3}. 
\end{align*}
Since $d,e,c_1$ and $c_3$ are disjoint, we have 
\[ (t_{c_1}t_{c_2}t_{c_3})^2t_{c_2}t_{c_1}t_{c_3}t_{c_2} = t_dt_{c_3}^{-1}t_et_{c_3}^{-1} \, . \]
\noindent Taking the $m$-th power of both sides, we obtain $T_{10m} = t_d^mt_{c_3}^{-m}t_e^mt_{c_3}^{-m}$ for any positive integer $m$, where $T_{10m} = \{(t_{c_1}t_{c_2}t_{c_3})^2t_{c_2}t_{c_1}t_{c_3}t_{c_2}\}^m$.
Now let 
\[ \phi_{12} = t_{c_4} t_{c_3} t_{c_2} t_{c_1} t_{c_1} t_{c_2} t_{c_3} t_{c_4} t_{c_4} t_d t_{c_3} t_{c_4} \, . \]
\noindent Since $\phi_{12}(c_3)=e$ and $\phi_{12}(d)=c_3$, we get:
\begin{align*}
\phi_{12}&=\phi_{12}t_d^{-m}t_{c_3}^mt_e^{-m}t_{c_3}^m T_{10m} \\
&=\phi_{12}t_d^{-m}t_{c_3}^m \phi_{12}^{-1} \phi_{12} t_e^{-m}t_{c_3}^m T_{10m} \\
&=t_{\phi_{12}(d)}^{-m}t_{\phi_{12}(c_3)}^m \phi_{12} t_e^{-m}t_{c_3}^m T_{10m} \\
&=t_{c_3}^{-m}t_e^m \phi_{12} t_e^{-m}t_{c_3}^m T_{10m}. 
\end{align*}
\noindent We thus obtain the commutator relation in \cite{BKM}
\[ C_m= [\phi_{12}, t_{c_3}^m t_e^m] = T_{10m} \, ,\]
the right-hand side of which contains arbitrarily long positive factorizations as $m$ increases. 

These commutator relations prescribe a family of genus-$2$ Lefschetz fibrations over $T^2$ with sections of self-intersection zero. Taking the complement of the regular fiber and the section, the first and the third authors of this article produced allowable Lefschetz fibrations filling a fixed spinal open book, leading to the first examples of contact $3$-manifolds with arbitrarily large Stein fillings and arbitrarily negative signatures \cite{BV1}. Guided by these examples, in a subsequent work \cite{BV2}, the same authors produced the first examples of mapping classes with arbitrarily long positive factorizations. They showed that \textit{any} family of commutators $C_m=[A_m, B_m]$ with arbitrarily long positive factorizations can be crafted into arbitrarily long positive factorizations of the boundary multitwist $t_{\delta_1} t_{\delta_2}$ in $\Gamma_g^2$, for $g \geq 8$. 

The arguments of \cite{BV2} were taken further in an elegant article by Dalyan, Korkmaz and Pamuk in \cite{DKP}, who observed that for special commutators $C_m=[A, B_m]$, where one entry is a fixed mapping class, as in the commutator relation we reproduced above, one can manipulate the relations so as to produce arbitrarily long positive factorizations in $\Gamma_2^2$. Namely, by repeating the relation (\ref{hurwitz2}), we have
\begin{align*}
(t_{c_1}t_{c_2}t_{c_3}t_{c_4}t_{c_5})^6 &= (t_{c_1}t_{c_2}t_{c_3}t_{c_4})^5 t_{c_5}t_{c_4}t_{c_3}t_{c_2}t_{c_1}t_{c_1}t_{c_2}t_{c_3}t_{c_4}t_{c_5} \\
&= t_{c_1}t_{c_2}t_{c_3}t_{c_4} (t_{c_1}t_{c_2}t_{c_3}t_{c_4})^4 t_{c_5}t_{c_4}t_{c_3}t_{c_2}t_{c_1}t_{c_1}t_{c_2}t_{c_3}t_{c_4}t_{c_5} \\
&= t_{c_1}t_{c_2}t_{c_3}t_{c_4} (t_{c_1}t_{c_2}t_{c_3})^4 t_{c_4}t_{c_3}t_{c_2}t_{c_1} t_{c_5}t_{c_4}t_{c_3}t_{c_2}t_{c_1}t_{c_1}t_{c_2}t_{c_3}t_{c_4}t_{c_5}.
\end{align*}
Since $t_\delta$ and $t_{\delta^\prime }$ are center elements of $\Gamma_2^2$, by the chain relations $t_\delta t_{\delta^\prime} = (t_{c_1}t_{c_2}t_{c_3}t_{c_4}t_{c_5})^6$ and $t_dt_e = (t_{c_1}t_{c_2}t_{c_3})^4$ we obtain 
\begin{align*}
t_\delta t_{\delta^\prime} &= t_{c_1}t_{c_2}t_{c_3}t_{c_4} t_dt_e t_{c_4}t_{c_3}t_{c_2}t_{c_1} t_{c_5} \cdot t_{c_4}t_{c_3}t_{c_2}t_{c_1}t_{c_1}t_{c_2}t_{c_3}t_{c_4} \cdot t_{c_5} \\
&= t_{c_4}t_{c_3}t_{c_2}t_{c_1}t_{c_1}t_{c_2}t_{c_3}t_{c_4} \cdot t_{c_5} \cdot t_{c_1}t_{c_2}t_{c_3}t_{c_4} t_dt_e t_{c_4}t_{c_3}t_{c_2}t_{c_1} t_{c_5} \\
&= t_{c_4}t_{c_3}t_{c_2}t_{c_1}t_{c_1}t_{c_2}t_{c_3}t_{c_4} \cdot t_{c_4}t_dt_{c_3} \cdot D_9 \\ 
&= D_9 \cdot t_{c_4}t_{c_3}t_{c_2}t_{c_1}t_{c_1}t_{c_2}t_{c_3}t_{c_4} \cdot t_{c_4}t_dt_{c_3}, 
\end{align*}
where $D_9=t_{(t_{c_4} t_d t_{c_3})^{-1}(c_5)} t_{c_1} t_{t_{c_3}^{-1}(c_2)} t_{(t_{c_4} t_d t_{c_3})^{-1}(c_3)} t_{e} t_{t_{c_3}^{-1}(c_4)} t_{c_2} t_{c_1} t_{c_5}$. 
By multiplying both sides of this relation by $t_{c_4}$, we obtain: 
\[ t_{\delta} t_{\delta^\prime} t_{c_4} = D_9 \cdot \phi_{12,m} \cdot T_{10m} \, .\]

\smallskip
We sum up these in the following:
\begin{thm}[\cite{BKM, DKP}] \label{DKP}
Let $d,e$ and $c_i$, $i=1,2,3,4,5$, be the simple closed curves on $\Sigma_2^2$ as in Figure~\ref{curves1}, and let 
\begin{align*}
\phi_{12} &= t_{c_4} t_{c_3} t_{c_2} t_{c_1} t_{c_1} t_{c_2} t_{c_3} t_{c_4} t_{c_4} t_d t_{c_3} t_{c_4}, \\
\phi_{12,m} &= t_{c_3}^{-m}t_e^m \phi_{12} t_e^{-m} t_{c_3}^m \\
T_{10m} &= \{(t_{c_1} t_{c_2} t_{c_3})^2 t_{c_2} t_{c_1} t_{c_3} t_{c_2}\}^m, \ \ \ \mathrm{and}\\
D_9 &= t_{(t_{c_4} t_d t_{c_3})^{-1}(c_5)} t_{c_1} t_{t_{c_3}^{-1}(c_2)} t_{(t_{c_4} t_d t_{c_3})^{-1}(c_3)} t_{e} t_{t_{c_3}^{-1}(c_4)} t_{c_2} t_{c_1} t_{c_5}. 
\end{align*}
Then, for all positive integer $m$, the following relations hold in $\Gamma_2^2$ \, :
\begin{align}
\phi_{12} &= \phi_{12,m} \cdot T_{10m}, \label{DKPrelation0} \ \ \ \ \ \ \ \ \ \text{(Baykur-Korkmaz-Monden)}
\\
t_{\delta} t_{\delta^\prime} t_{c_4} &= D_9 \cdot \phi_{12,m} \cdot T_{10m}. \ \ \ \ \text{(Dalyan-Korkmaz-Pamuk)} \label{DKPrelation1} 
\end{align}
\end{thm}

\begin{figure}[ht]
\centering
\includegraphics[scale=.45]{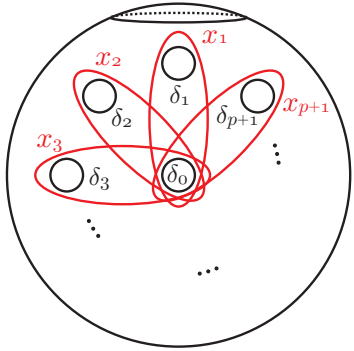}
\caption{The curves $\delta_0,\delta_1,\ldots,\delta_{p+1}$ and $x_1,x_2,\ldots,x_{p+1}$.}
\label{daisy}
\end{figure}
Finally, let us recall the following generalization of the \textit{lantern relation}, which was proved in \cite{PVHM}, and independently in \cite{BKM}, and was called the \textit{daisy relation} in \cite{EMVHM}. This relation will be key for inflating number of boundary components, and extending Theorem~\ref{DKP} to $\Gamma_g^n$ for any $1 \leq n \leq 2g-4$.
\begin{align*}
t_{\delta_0}^{p-1} t_{\delta_1} t_{\delta_2} \cdots t_{\delta_{p+1}} = t_{x_1} t_{x_2} \cdots t_{x_{p+1}}.
\end{align*}
\noindent in $\Gamma_0^{p+2}$, the mapping class group of a $2$-sphere with $p+2 \geq 4$ boundary components. Here $\delta_0, \delta_1, \delta_2,\ldots, \delta_{p+1}$ denote the $p+2$ boundary curves of $\Sigma_0^{p+2}$, and $x_1, x_2,\ldots, x_{p+1}$ are the interior curves as shown in Figure~\ref{daisy}. The $p=2$ case is the usual lantern relation.

\subsection{Boundary multitwist of infinite length} \

\begin{thm}\label{thm1}
Let $g\geq 3$. Then, in $\Gamma_g^{2g-4}$, the multitwist 
\begin{align*}
t_{\delta_1} t_{\delta_2} \cdots t_{\delta_{2g-4}}
\end{align*}
can be written as a product of arbitrarily large number of right-handed Dehn twists about nonseparating curves. 
\end{thm}

Let $a,b,d,d^\prime,e,e^\prime$ and $c_i$ $(i=1,2,\ldots,2g+1)$ be the simple closed curves on $\Sigma_g^2$, 
and let $\delta$ and $\delta^\prime$ be the two boundary curves of $\Sigma_g^2$ as in Figure~\ref{curves1}.

We will now introduce the key lemma for the proofs of Theorem~\ref{thm1} and~\ref{thm2}, the Lemma~\ref{sections1}.

Let $l$ be a positive integer such that $l\leq n$. Let $\beta$ and $\alpha,\alpha^\prime$ be the separating curve and the nonseparating curves on $\Sigma_g^n$ in Figure~\ref{idea}, respectively. Note that $\beta$ separates $\Sigma_g^n$ into a surface of genus $g$ with one boundary $\beta$ and a sphere with $l+1$ boundaries $d, \delta_1, \delta_2, \ldots, \delta_l$ and that $\alpha$ and $\alpha^\prime$ separate $\Sigma_g^n$ into a surface of genus $g-1$ with $2$ boundaries $\alpha$ and $\alpha^\prime$ and a sphere with $l+2$ boundaries $\alpha, \alpha^\prime, \delta_1,\delta_2,\ldots,\delta_l$. Let $x_1,x_2,\ldots,x_l$ be the nonseparating curves on $\Sigma_g^n$ in Figure~\ref{idea}. 

\begin{figure}[h]
 \centering
      \includegraphics[scale=.80]{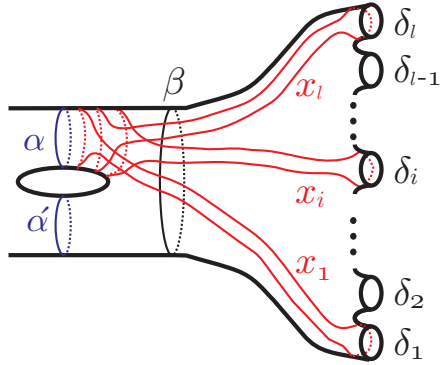}
      \caption{The curves $\alpha, \alpha^\prime,\beta$ and $x_1,x_2,\ldots,x_l$.}
      \label{idea}
\end{figure}

\begin{lemma}\label{sections1}
Suppose that the following relation holds in $\Gamma_g^n$:
\begin{align*}
U \cdot t_\beta = T \cdot t_\alpha^{l-1} t_{\alpha^\prime},
\end{align*}
where $U$ and $T$ are elements in $\Gamma_g^n$. 
Then, the following relation holds in $\Gamma_g^n$:
\begin{align*}
U \cdot t_{\delta_1} t_{\delta_2} \cdots t_{\delta_l}=T \cdot t_{x_1}\cdots t_{x_l}.
\end{align*}
\end{lemma}

\noindent This is a generalization of a technical lemma from \cite{BM}, which we provide a different proof of below.

\begin{proof}
We multiply both sides of the relation $U \cdot t_\beta = T \cdot t_\alpha^{l-1} t_{\alpha^\prime}$ by $\delta_1\delta_2\cdots \delta_l$, we obtain the following relation:
\begin{align*}
U \cdot t_\beta t_{\delta_1} t_{\delta_2} \cdots t_{\delta_l} = T \cdot t_\alpha^{l-1} t_{\alpha^\prime} t_{\delta_1} t_{\delta_2} \cdots t_{\delta_l}.
\end{align*}
Since $t_{\delta_1}, t_{\delta_2}, \ldots, t_{\delta_l}$ are elements in the center of $\Gamma_g^n$, we can rewrite this relation as follows:
\begin{align*}
U \cdot t_{\delta_1} t_{\delta_2} \cdots t_{\delta_l} t_\beta = T \cdot t_\alpha^{l-1} t_{\delta_1} t_{\delta_2} \cdots t_{\delta_l} t_{\alpha^\prime}.
\end{align*}
Here, by the daisy relation $t_\alpha^{l-1} t_{\delta_1} t_{\delta_2} \cdots t_{\delta_l} t_{\alpha^\prime} = t_{x_1} t_{x_2} \cdots t_{x_l} t_\beta$, we have 
\begin{align*}
U \cdot t_{\delta_1} t_{\delta_2} \cdots t_{\delta_l} t_\beta = T \cdot t_{x_1} t_{x_2} \cdots t_{x_l} t_\beta.
\end{align*}
Removing $t_\beta$ from both sides of this relation we get the desired relation.
\end{proof}

Let $d_j$, $e_j$ $(j=4,5,\ldots,2g+1)$, $f_h$ $(h=6,7,8,9)$ be the simple closed curves on $\Sigma_g^2$ as in Figure~\ref{curves7} which are defined by 
\begin{align*}
&d_j=t_{c_{j-3}}^{-1}t_{c_{j-2}}^{-1}t_{c_{j-1}}^{-1}(c_j),& 
&e_j=t_{c_{j-3}} t_{c_{j-2}} t_{c_{j-1}}(c_j),& \\
&f_h=t_{c_{h-5}}t_{c_{h-4}}t_{c_{h-3}}t_{c_{h-2}}t_{c_{h-1}}(c_h).& 
\end{align*}

\begin{figure}[ht]
\centering
\includegraphics[scale=.70]{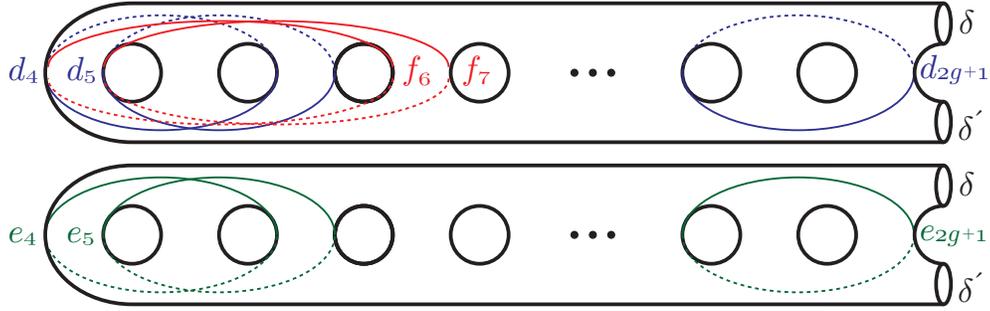}
\caption{The curves $d_j,e_j$ $(j=4,5,\ldots,2g+1)$ and $f_h$ $(h=6,7,8,9)$ on $\Sigma_g^2$.}
\label{curves7}
\end{figure}

Let $i,l,m$ be positive integers such that $l+1 \leq i \leq m-1$, the following relations hold from the braid relations:
\begin{align}
&t_{c_{i-1}} \cdot t_{c_m} t_{c_{m-1}} \cdots t_{c_l} = t_{c_m} t_{c_{m-1}} \cdots t_{c_l} \cdot t_{c_i}, \label{hurwitz1}\\
&t_{c_l} t_{c_{l+1}} \cdots t_{c_m} \cdot t_{c_{i-1}} = t_{c_i} \cdot t_{c_l} t_{c_{l+1}} \cdots t_{c_m}. \label{hurwitz2}
\end{align}

Next is a lemma from \cite{BM}, which we include the proof of here for completeness.

\begin{lemma}[\cite{BM}]
For $k=1,2,\ldots,2g-2$, the following relations hold in $\Sigma_g^2$: 
\begin{align}
&\prod_{i=k}^{2g-2} t_{c_{i+3}} t_{c_{i+2}} t_{c_{i+1}} t_{c_i} = (t_{c_{k+2}} t_{c_{k+1}} t_{c_k})^{2g-1-k} t_{d_{k+3}} t_{d_{k+4}} \cdots t_{d_{2g+1}}, \label{equationa} \\
&\prod_{i=2g-2}^k t_{c_i} t_{c_{i+1}} t_{c_{i+2}} t_{c_{i+3}} = t_{e_{2g+1}} \cdots t_{e_{k+4}} t_{e_{k+3}} (t_{c_k} t_{c_{k+1}} t_{c_{k+2}})^{2g-1-k}, \label{equationb} \\
&\prod_{i=4}^1 t_{c_i} t_{c_{i+1}} t_{c_{i+2}} t_{c_{i+3}} t_{c_{i+4}} t_{c_{i+5}} = t_{f_9} t_{f_8} t_{f_7} t_{f_6} (t_{c_1} t_{c_2} t_{c_3} t_{c_4} t_{c_5})^4. \label{equationc} 
\end{align}
\end{lemma}

\begin{proof}
First, we prove the relation (\ref{equationa}) by induction on $2g-1-k$. Suppose that $k=2g-2$. Then, we have 
\begin{align*}
t_{c_{2g+1}} \cdot t_{c_{2g}} t_{c_{2g-1}} t_{c_{2g-2}} = t_{c_{2g}} t_{c_{2g-1}} t_{c_{2g-2}} \cdot t_{d_{2g+1}}.
\end{align*}
Therefore, the conclusion of the relation holds for $k=1$. Let us assume, inductively, that the relation holds for $k+1<2g-2$. By (\ref{hurwitz1}), we have 
\begin{align*}
&\prod_{i=k}^{2g-2} t_{c_{i+3}} t_{c_{i+2}} t_{c_{i+1}} t_{c_i} = t_{c_{k+3}} t_{c_{k+2}} t_{c_{k+1}} t_{c_k} \cdot \prod_{i=k+1}^{2g-2} t_{c_{i+3}} t_{c_{i+2}} t_{c_{i+1}} t_{c_i} \\
&= t_{c_{k+3}} t_{c_{k+2}} t_{c_{k+1}} t_{c_k} \cdot (t_{c_{k+3}} t_{c_{k+2}} t_{c_{k+1}})^{2g-1-(k+1)} t_{d_{k+4}} t_{d_{k+5}} \cdots t_{d_{2g+1}} \\
&= (t_{c_{k+2}} t_{c_{k+1}} t_{c_k})^{2g-1-(k+1)} \cdot t_{c_{k+3}} t_{c_{k+2}} t_{c_{k+1}} t_{c_k} \cdot t_{d_{k+4}} t_{d_{k+5}} \cdots t_{d_{2g+1}} \\
&= (t_{c_{k+2}} t_{c_{k+1}} t_{c_k})^{2g-1-(k+1)} \cdot t_{c_{k+2}} t_{c_{k+1}} t_{c_k} \cdot t_{d_{k+3}} \cdot t_{d_{k+4}} t_{d_{k+5}} \cdots t_{d_{2g+1}}. 
\end{align*}
Hence, the relation (\ref{equationa}) is proved. 

Next, we prove the relation (\ref{equationb}) by induction on $2g-1-k$. Suppose that $k=2g-2$. Then, we have 
\begin{align*}
t_{c_{2g-2}} t_{c_{2g-1}} t_{c_{2g}} \cdot t_{c_{2g+1}} = t_{e_{2g+1}} \cdot t_{c_{2g-2}} t_{c_{2g-1}} t_{c_{2g}}. 
\end{align*}
Therefore, the conclusion of the relation holds for $k=2g-2$. Let us assume, inductively, that the relation holds for $k+1<2g-2$. By (\ref{hurwitz2}), we have 
\begin{align*}
&\prod_{i=2g-2}^k t_{c_i} t_{c_{i+1}} t_{c_{i+2}} t_{c_{i+3}} = \left(\prod_{i=2g-2}^{k+1} t_{c_i} t_{c_{i+1}} t_{c_{i+2}} t_{c_{i+3}}\right) \cdot t_{c_k} t_{c_{k+1}} t_{c_{k+2}} t_{c_{k+3}} \\
& = t_{e_{2g+1}} \cdots t_{e_{k+5}} t_{e_{k+4}} (t_{c_{k+1}} t_{c_{k+2}} t_{c_{k+3}})^{2g-2-k} \cdot t_{c_k} t_{c_{k+1}} t_{c_{k+2}} t_{c_{k+3}} \\
& = t_{e_{2g+1}} \cdots t_{e_{k+5}} t_{e_{k+4}} \cdot t_{c_k} t_{c_{k+1}} t_{c_{k+2}} t_{c_{k+3}} \cdot (t_{c_k} t_{c_{k+1}} t_{c_{k+2}})^{2g-2-k} \\
& = t_{e_{2g+1}} \cdots t_{e_{k+5}} t_{e_{k+4}} \cdot t_{e_{k+3}} \cdot t_{c_k} t_{c_{k+1}} t_{c_{k+2}} \cdot (t_{c_k} t_{c_{k+1}} t_{c_{k+2}})^{2g-2-k}. 
\end{align*}
Hence we obtain the relation~(\ref{equationb}). 
The proof for the relation (\ref{equationc}) is very similar, and we will leave it to the reader.
\end{proof}

Let
\begin{align*}
&H_i:=t_{c_1} t_{c_2} \cdots t_{c_i},& && &\overline{H}_i:=t_{c_i} \cdots t_{c_2} t_{c_1}& \\
&I_{2g-8}:=t_{d_{10}} t_{d_{11}} \cdots t_{d_{2g+1}},& && &J_{2g-6}:=t_{e_{2g+1}} \cdots t_{e_9} t_{e_8},& \\
&K_4:=t_{f_9} t_{f_8} t_{f_7} t_{f_6},& &\mathrm{and}& & L_{16}:=\prod_{i=1}^4 t_{c_{i+3}} t_{c_{i+2}} t_{c_{i+1}} t_{c_i}.& 
\end{align*}

\begin{lemma}\label{lem1}
For $g\geq 4$, the following relation holds in $\Sigma_g^2$:
\begin{align*}
(t_{c_1} t_{c_2} \cdots t_{c_{2g+1}})^4 = \left(H_3\right)^4 \left(\overline{H}_3\right)^{2g-8} K_4 \left(H_5\right)^4 I_{2g-8}. 
\end{align*}
\end{lemma}

\begin{proof}
It is easy to check that from the braid relations we have 
\begin{align*}
(t_{c_1} t_{c_2} \cdots t_{c_{2g+1}})^4 &= (t_{c_1} t_{c_2} t_{c_3})^4 \prod_{i=4}^1 t_{c_i} \cdots t_{c_{i+2g-3}} \\
&= \left(H_3\right)^4 \prod_{i=1}^{2g-2} t_{c_{i+3}} t_{c_{i+2}} t_{c_{i+1}} t_{c_i}. 
\end{align*}
By the relation (\ref{equationa}) for $k=7$ repeating (\ref{hurwitz1}), we obtain
\begin{align*}
(t_{c_1} t_{c_2} \cdots t_{c_{2g+1}})^4 &= \left(H_3\right)^4  \left(\prod_{i=1}^{6} t_{c_{i+3}} t_{c_{i+2}} t_{c_{i+1}} t_{c_i}\right) (t_{c_9} t_{c_8} t_{c_7})^{2g-8} I_{2g-8} \\ 
&= \left(H_3\right)^4 (t_{c_3} t_{c_2} t_{c_1})^{2g-8} \left(\prod_{i=1}^{6} t_{c_{i+3}} t_{c_{i+2}} t_{c_{i+1}} t_{c_i}\right) I_{2g-8}. 
\end{align*}
Since it is easy to check that from the braid relations we have 
\begin{align*}
\prod_{i=1}^{6} t_{c_{i+3}} t_{c_{i+2}} t_{c_{i+1}} t_{c_i} = \prod_{i=4}^1 t_{c_i} t_{c_{i+1}} t_{c_{i+2}} t_{c_{i+3}} t_{c_{i+4}} t_{c_{i+5}}, 
\end{align*}
by (\ref{equationc}), we obtain the desired relation.
\end{proof}

\begin{lemma}\label{lem2}
For $g\geq 4$, the following relation holds in $\Sigma_g^2$:
\begin{align*}
(t_{c_1} t_{c_2} \cdots t_{c_{2g+1}})^{2g-2} &= J_{2g-6} L_{16} \left(H_3\right)^{2g-6} t_{d^\prime} t_{e^\prime}. 
\end{align*}
\end{lemma}

\begin{proof}
From the braid relations we have
\begin{align*}
&(t_{c_1}t_{c_2}  \cdots t_{c_{2g+1}})^{2g-2} = \left(\prod_{i=2g-2}^1 t_{c_i} t_{c_{i+1}} t_{c_{i+2}} t_{c_{i+3}} \right) (t_{c_5} t_{c_6} \cdots t_{c_{2g+1}})^{2g-2}. 
\end{align*}
By the chain relation $t_{d^\prime} t_{e^\prime} =(t_{c_5}t_{c_6}\cdots t_{c_{2g+1}})^{2g-2}$, the relation (\ref{equationb}) for $k=5$ and repeating (\ref{hurwitz2}), 
\begin{align*}
(t_{c_1} t_{c_2}  \cdots t_{c_{2g+1}})^{2g-2} &= J_{2g-6} (t_{c_5} t_{c_6} t_{c_7})^{2g-6} \left( \prod_{i=4}^1 t_{c_i} t_{c_{i+1}} t_{c_{i+2}} t_{c_{i+3}} \right) t_{d^\prime} t_{e^\prime} \\
&= J_{2g-6} \left( \prod_{i=4}^1 t_{c_i} t_{c_{i+1}} t_{c_{i+2}} t_{c_{i+3}} \right) (t_{c_1} t_{c_2} t_{c_3})^{2g-6} t_{d^\prime} t_{e^\prime}.
\end{align*}
Since it is easy to check that from the braid relations we have 
\begin{align*}
\prod_{i=4}^1 t_{c_i} t_{c_{i+1}} t_{c_{i+2}} t_{c_{i+3}} = \prod_{i=1}^4 t_{c_{i+3}} t_{c_{i+2}} t_{c_{i+1}} t_{c_i}, 
\end{align*}
we obtain the relation in the statement.
\end{proof}

\begin{proposition}\label{prop1}
Let $g\geq 4$. 
Then, the following relation holds in $\Sigma_g^2$:
If $g$ is even, then we have 
\begin{align*}
t_\delta t_{\delta^\prime} &= K_4 \left(H_5\right)^4 I_{2g-8} J_{2g-6} L_{16} \left(H_3\right)^2 t_d^{g-3} t_{d^\prime} t_e^{g-3} t_{e^\prime}.
\end{align*}
If $g$ is odd, then we have 
\begin{align*}
t_\delta t_{\delta^\prime} &= K_4 \left(H_5\right)^4 I_{2g-8} J_{2g-6} L_{16} \left(\overline{H}_3\right)^2 t_d^{g-3} t_{d^\prime} t_e^{g-3} t_{e^\prime}.
\end{align*}
\end{proposition}

\begin{proof}
By Lemma~\ref{lem1},~\ref{lem2} and the chain relation 
\begin{align*}
t_\delta t_{\delta^\prime} = (t_{c_1} t_{c_2} \cdots t_{c_{2g+1}})^{2g+2}=(t_{c_1} t_{c_2} \cdots t_{c_{2g+1}})^4 \cdot (t_{c_1} t_{c_2} \cdots t_{c_{2g+1}})^{2g-2}, 
\end{align*}
we have 
\begin{align*}
&t_\delta t_{\delta^\prime} = \left(H_3\right)^4 \left(\overline{H}_3\right)^{2g-8} K_4 \left(H_5\right)^4 I_{2g-8} J_{2g-6} L_{16} \left(H_3\right)^{2g-6} t_{d^\prime} t_{e^\prime}. 
\end{align*}
Since $c_1,c_2,c_3$ are disjoint from $d^\prime$ and $e^\prime$, by conjugation by $\left(H_3\right)^4 \left(\overline{H}_3\right)^{2g-8}$ we obtain 
\begin{align*}
t_\delta t_{\delta^\prime} &= K_4 \left(H_5\right)^4 I_{2g-8} J_{2g-6} L_{16} \left(H_3\right)^{2g-2} \left(\overline{H}_3\right)^{2g-8} t_{d^\prime} t_{e^\prime}. 
\end{align*}
The claim follows from this relation and the chain relations $t_d t_{e} = \left(H_3\right)^4 = \left(\overline{H}_3\right)^4$. 
\end{proof}

{
\begin{lemma}\label{long1}
Let $g\geq 4$. For any positive integer $m$, we have 
\begin{align*}
L_{16} \left(H_5\right)^4 = \phi_{12,m} T_{10m} M_9 \cdot t_{c_5} t_{c_3} t_{c_4} t_{c_2} t_{c_3},
\end{align*}
where $M_9 = t_{(t_{c_4} t_d t_{c_3} t_{c_4})^{-1}(c_5)} t_{c_2} t_{(t_{c_5}t_{c_4})^{-1}(c_6)} t_d t_{t_{c_4}^{-1}(c_3)} t_{c_7} t_{(t_{c_5}t_{c_4})^{-1}(c_6)} t_d t_{t_{c_4}^{-1}(e)}$. 
\end{lemma}
\begin{proof}
By (\ref{hurwitz2}) and the braid relations, we have
\begin{align*}
L_{16} \left(H_5\right)^4 & = t_{c_4} t_{c_3} t_{c_2} t_{c_1} H_5 t_{c_4} t_{c_3} t_{c_2} t_{c_1} t_{t_{c_5}^{-1}(c_6)} t_{c_4} t_{c_3} t_{c_2} t_{c_7} t_{t_{c_5}^{-1}(c_6)} t_{c_4} t_{c_3} \left(H_5\right)^3 \\
& = t_{c_4} t_{c_3} t_{c_2} t_{c_1} H_5 t_{c_4} t_{c_3} t_{c_2} t_{t_{c_5}^{-1}(c_6)} t_{c_4} t_{c_3} t_{c_7} t_{t_{c_5}^{-1}(c_6)} t_{c_4} \cdot t_{c_1} t_{c_2} t_{c_3} \left(H_5\right)^3 \\
& = t_{c_4} t_{c_3} t_{c_2} t_{c_1} H_5 t_{c_4} t_{c_3} t_{c_2} t_{t_{c_5}^{-1}(c_6)} t_{c_4} t_{c_3} t_{c_7} t_{t_{c_5}^{-1}(c_6)} t_{c_4} \cdot (H_3)^4 t_{c_4} t_{c_5} t_{c_3} t_{c_4} t_{c_2} t_{c_3} .
\end{align*}
Here, by the chain relation $t_d t_e = (H_3)^4$, 
\begin{align*}
t_{c_2} t_{t_{c_5}^{-1}(c_6)} t_{c_4} t_{c_3} t_{c_7} t_{t_{c_5}^{-1}(c_6)} t_{c_4} (H_3)^4 t_{c_4} &= t_{c_2} t_{t_{c_5}^{-1}(c_6)} t_{c_4} t_{c_3} t_{c_7} t_{t_{c_5}^{-1}(c_6)} t_{c_4} t_d t_e t_{c_4} \\
&= t_{c_2} t_{t_{c_5}^{-1}(c_6)} t_{c_4} t_{c_3} t_{c_7} t_{t_{c_5}^{-1}(c_6)} t_{c_4} t_d t_{c_4} t_{t_{c_4}^{-1}(e)} \\
&= t_d t_{c_4} \cdot N_7 \cdot t_{c_4} t_{t_{c_4}^{-1}(e)}, 
\end{align*}
where $N_7 =  (t_d t_{c_4})^{-1} (t_{c_2} t_{t_{c_5}^{-1}(c_6)} t_{c_4} t_{c_3} t_{c_7} t_{t_{c_5}^{-1}(c_6)} t_{c_4}) (t_d t_{c_4})$. 
Note that it is easy to check that $N_7 = t_{c_2} t_{(t_{c_5}t_{c_4})^{-1}(c_6)} t_d t_{t_{c_4}^{-1}(c_3)} t_{c_7} t_{(t_{c_5}t_{c_4})^{-1}(c_6)} t_d$. 
Therefore, we have
\begin{align*}
L_{16} \left(H_5\right)^4 &= t_{c_4} t_{c_3} t_{c_2} t_{c_1} H_5 t_{c_4} t_{c_3} \cdot t_d t_{c_4} \cdot N_7 \cdot t_{c_4} t_{t_{c_4}^{-1}(e)} \cdot t_{c_5} t_{c_3} t_{c_4} t_{c_2} t_{c_3} \\
&= t_{c_4} t_{c_3} t_{c_2} t_{c_1} \cdot t_{c_1} t_{c_2} t_{c_3} t_{c_4} t_{c_5} \cdot t_{c_4} t_d t_{c_3} t_{c_4} \cdot N_7 \cdot t_{c_4} t_{t_{c_4}^{-1}(e)} \cdot t_{c_5} t_{c_3} t_{c_4} t_{c_2} t_{c_3} \\
&= t_{c_4} t_{c_3} t_{c_2} t_{c_1} \cdot t_{c_1} t_{c_2} t_{c_3} t_{c_4} \cdot t_{c_4} t_d t_{c_3} t_{c_4} \cdot M_9 \cdot t_{c_5} t_{c_3} t_{c_4} t_{c_2} t_{c_3} \\
&= \phi_{12} \cdot M_9 \cdot t_{c_5} t_{c_3} t_{c_4} t_{c_2} t_{c_3}. 
\end{align*}
The Lemma follows from the relation (\ref{DKPrelation0}). 
\end{proof}}

\begin{proof}[Proof of Theorem~\ref{thm1}]
Suppose that $g\geq 4$. 
Let $S$ and $S^\prime$ be two spheres with $g-1$ boundary components, and we denoted by $\delta,\delta_1,\delta_2,\ldots,\delta_{g-2}$ and $\delta^\prime,\delta_{g-1},\delta_g,\ldots,\delta_{2g-4}$ the boundary curves of $S$ and $S^\prime$, respectively. 
We attach $S$ and $S^\prime$ to $\Sigma_g^2$ along $\delta$ and $\delta^\prime$. 
Then, we obtain a compact oriented surface of genus $g$ with $2g-4$ boundary components $\delta_1,\delta_2,\ldots,\delta_{2g-4}$, denoted by $\Sigma_g^{2g-4}$. 
By Proposition~\ref{prop1} and Lemma~\ref{sections1}, there are simple close curves $x_1,x_2,\ldots,x_{2g-4}$ such that the following relations hold in $\Gamma_g^{2g-4}$:
Let 
\begin{align*}
&Z_{g-2}:=t_{x_1} t_{x_2} \cdots t_{x_{g-2}}& &\mathrm{and}& &W_{g-2}:=t_{x_{g-1}} t_{x_g} \cdots t_{x_{2g-4}}.&
\end{align*}
{If $g$ is even, then we have 
\begin{align*}
t_{\delta_1} t_{\delta_2} \cdots t_{\delta_{2g-4}} = K_4 \left(H_5\right)^4 I_{2g-8} J_{2g-6} L_{16} \left(H_3\right)^2 Z_{g-2} W_{g-2}.
\end{align*}
If $g$ is odd, then we have 
\begin{align*}
t_{\delta_1} t_{\delta_2} \cdots t_{\delta_{2g-4}} =  K_4 \left(H_5\right)^4 I_{2g-8} J_{2g-6} L_{16} \left(\overline{H}_3\right)^2 Z_{g-2} W_{g-2}.
\end{align*}
By conjugation by $L_{16}$ and Lemma~\ref{long1}, we have the following relation: 
If $g$ is even, then we have 
\begin{align*}
&t_{\delta_1} t_{\delta_2} \cdots t_{\delta_{2g-4}} = K_4^\prime \phi_{12,m} T_{10m} M_9 \cdot t_{c_5} t_{c_3} t_{c_4} t_{c_2} t_{c_3} \cdot I_{2g-8} J_{2g-6} \left(H_3^\prime\right)^2 Z_{g-2}^\prime W_{g-2}^\prime, 
\end{align*}
where $K^\prime_4 = L_{16} K_4 L_{16}^{-1}$, $H_3^\prime = L_{16} H_3 L_{16}^{-1}$ $Z^\prime_{g-2}= L_{16} Z_{g-2} L_{16}^{-1}$ and $W^\prime_{g-2} = L_{16} W_{g-2} L_{16}^{-1}$. 
If $g$ is odd, then we have 
\begin{align*}
&t_{\delta_1} t_{\delta_2} \cdots t_{\delta_{2g-4}} = K_4^\prime \phi_{12,m} T_{10m} M_9 \cdot t_{c_5} t_{c_3} t_{c_4} t_{c_2} t_{c_3} \cdot I_{2g-8} J_{2g-6} \left(\overline{H^\prime}_3\right)^2 Z_{g-2}^\prime W_{g-2}^\prime, 
\end{align*}
where $\overline{H^\prime}_3 = L_{16} \overline{H}_3 L_{16}^{-1}$. }

{
Note that $K^\prime_4$, $H^\prime_3$, $\overline{H^\prime}_3$ , $Z^\prime_{g-2}$ and $W^\prime_{g-2}$ are also products of $4$, $3$, $3$, $g-2$ and $g-2$ right-handed Dehn twists about nonseparating curves, respectively. 
Therefore, for any positive integer $m$, $t_{\delta_1} t_{\delta_2} \cdots t_{\delta_{2g-4}}$ may can be written as a product of $6g+2+10m$ right-handed Dehn twists about nonseparating curves. This completes the proof. }
\end{proof}

\subsection{Factorizations of boundary multitwist and a single Dehn twist} 

\begin{thm}\label{thm2}
Let $g\geq 2$. 
Let $a$ be a nonseparating curve on $\Sigma_g^n$. 
Then, for any positive integer $n$, in the mapping class group $\Gamma_g^n$, the multitwist 
\begin{align*}
t_{\delta_1} t_{\delta_2} \cdots t_{\delta_n} t_a 
\end{align*}
can be written as a product of arbitrarily large number of right-handed Dehn twists about nonseparating curves. 
\end{thm}

Let $O_s, P_s$ $Q_s$ and $R_s$ be products of $s$ right-handed Dehn twists about nonseparating curves. 
Note that for any $\Phi$ in $\Gamma_g^n$, $\Phi O_s \Phi^{-1}, \Phi P_s \Phi^{-1}$, $\Phi Q_s \Phi^{-1}$ and $\Phi R_s \Phi^{-1}$ are also products of $s$ right-handed Dehn twists about nonseparating curves.

\begin{proof}[Proof of Theorem~\ref{thm2}]
Suppose that $n\geq 2$. 
Let $k$ be a positive integer. 
Note that, if $2m>4(n-2)+2$, then by the chain relation $t_d t_e = (t_{c_1} t_{c_2} t_{c_3})^4$, we can rewrite $T_{10m}= \{(t_{c_1} t_{c_2} t_{c_3})^2 t_{c_2} t_{c_1} t_{c_3} t_{c_2}\}^m$ in $\Gamma_2^2$ as follows:
\begin{align}
T_{10m} &= H_{10m-12n+24} \cdot t_{c_1}^2 (t_{c_1} t_{c_2} t_{c_3})^{4(n-2)} \notag \\
&= O_{10m-12n+22} \cdot t_{c_1}^2 t_d^{n-2} t_e^{n-2}. \label{H}
\end{align}

Let $S$ be a sphere with $n$ boundary components $\delta,\delta_1,\delta_2,\ldots,\delta_{n-1}$. We attach $S$ to $\Sigma_g^2$ along $\delta$. Set $\delta^\prime=\delta_n$. Then, we obtain a compact oriented surface of genus $g$ with $n$ boundary components $\delta_1,\delta_2,\ldots,\delta_n$, denoted by $\Sigma_g^n$. 

Suppose that $g=2$. Let $D_8 = D_9 t_{c_5}^{-1}$. By the relations (\ref{DKPrelation1}) and (\ref{H}), we have 
\begin{align*}
t_{\delta} t_{\delta_n} t_{c_4} &= D_8 \cdot t_{c_5} \cdot \phi_{12,m} \cdot O_{10m-12n+22} \cdot t_{c_1}^2 t_d^{n-2} t_e^{n-2} \\
&= D_8 \cdot \phi^\prime_{12,m} \cdot O_{10m-12n+22}^\prime \cdot t_{c_1}^2 t_e^{n-2} t_d^{n-2} t_{c_5}, 
\end{align*}
where $\phi^\prime_{12,m} = t_{c_5} \phi_{12,m} t_{c_5}^{-1}$, and $O_{10m-12n+22}^\prime = t_{c_5} O_{10m-12n+22} t_{c_5}^{-1}$. By Lemma~\ref{sections1}, there are nonseparating curves $x_1,\ldots,x_{n-1}$ on $\Sigma_2^n$ such that the following relation holds in $\Gamma_2^n$:
\begin{align}
t_{\delta_1} \cdots t_{\delta_{n-1}} t_{\delta_n} t_{c_4} = D_8 \cdot \phi^\prime_{12,m} \cdot O_{10m-12n+22}^\prime \cdot t_{c_1}^2 t_e^{n-2} \cdot x_1 x_2 \cdots x_{n-1}. \label{g2}
\end{align}
Therefore, if $g=2$, then the element $t_{\delta_1} \cdots t_{\delta_{n-1}} t_{\delta_n} t_{c_4}$ can be written as a product of $10m-10n+41$ right-handed Dehn twists about nonseparating simple closed curves for any $m>2(n-2)+2$. 

Suppose that $g\geq 3$. 
Let $a$ and $b$ be the simple closed curves on $\Sigma_g^2$ as in Figure~\ref{curves1}. 
By the chain relations $t_{\delta} t_{\delta_n} = (t_{c_1} t_{c_2} \cdots t_{c_{2g+1}})^{2g+2} = (t_{c_1} t_{c_2} \cdots t_{c_5})^6 t_{c_4} \cdot P_{4g^2+6g-29}$ and ${t_a t_b} = (t_{c_1} t_{c_2} \cdots t_{c_5})^6$, we obtain the following relation: 
\begin{align*}
t_{\delta} t_{\delta_n} = {t_a t_b} t_{c_4} \cdot P_{4g^2+6g-29}. 
\end{align*}
Therefore, by the relation (\ref{DKPrelation1}) and (\ref{H}), we have 
\begin{align}
t_{\delta} t_{\delta_n} &= D_9 \cdot \phi_{12,m} \cdot O_{10m-12n+22} \cdot t_{c_1}^2 t_d^{n-2} t_e^{n-2} \cdot P_{4g^2+6g-29}, \notag \\
&= D_9 \cdot \phi_{12,m} \cdot O_{10m-12n+22} \cdot t_{c_1}^2 t_e^{n-2} \cdot P_{4g^2+6g-29}^\prime \cdot t_d^{n-2}, \label{relation3}
\end{align}
where $P_{4g^2+6g-29}^\prime = t_d^{n-2} P_{4g^2+6g-29} t_d^{-n+2}$. Let us consider the simple closed curve $d^\prime$ on $\Sigma_g^2$ as in Figure~\ref{curves1}. When we multiply both sides of (\ref{relation3}) by $t_{d^\prime}$, we obtain the following relation:
\begin{align*}
t_{\delta} t_{\delta_n} t_{d^\prime} = D_9 \cdot \phi_{12,m} \cdot O_{10m-12n+22} \cdot t_{c_1}^2 t_e^{n-2} \cdot P_{4g^2+6g-29}^\prime \cdot t_d^{n-2} t_{d^\prime}. 
\end{align*}
By Lemma~\ref{sections1}, there are nonseparating curves $x_1,\ldots,x_{n-1}$ on $\Sigma_g^n$ such that the following relation holds in $\Gamma_g^n$:
\begin{align}
t_{\delta_1} \cdots t_{\delta_n} t_{d^\prime} = D_9 \cdot \phi_{12,m} \cdot O_{10m-12n+22} \cdot t_{c_1}^2 t_e^{n-2} \cdot P_{4g^2+6g-29}^\prime \cdot x_1 x_2 \cdots x_{n-1}. \label{relation4}
\end{align}
Therefore, if $g\geq 3$, then the element $t_{\delta_1} \cdots t_{\delta_{n-1}} t_{\delta_n} t_{d^\prime}$ can be written as a product of $4g^2+6g+13+10m-10n$ right-handed Dehn twists about nonseparating simple closed curves for any $m>2(n-2)+2$. 
\end{proof}

\subsection{Powers of boundary multitwists have infinite length} 

\begin{thm}\label{thm3}
Let $g\geq 2$, and let $k \geq 2$ be a positive integer. 
Then, for any $k$ and $n$, in the mapping class group $\Gamma_g^n$, the element 
\begin{align*}
(t_{\delta_1} t_{\delta_2} \cdots t_{\delta_n})^k
\end{align*}
can be written as a product of arbitrarily large number of right-handed Dehn twists about nonseparating curves. 
\end{thm}

\begin{proof}[Proof of Theorem~\ref{thm3}] 
Suppose that $k\in\{2,3\}$ and $n\geq 2$. 
Let $S$ be a sphere with $n$ boundary components $\delta,\delta_1,\delta_2,\ldots,\delta_{n-1}$. 
We attach $S$ to $\Sigma_g^2$ along $\delta$. 
Set $\delta^\prime=\delta_n$. 
Then, we obtain a compact oriented surface of genus $g$ with $n$ boundary components $\delta_1,\delta_2,\ldots,\delta_n$, denoted by $\Sigma_g^n$. 

Suppose that $g=2$. For simplicity, we write the relation (\ref{g2}) as follows:
\begin{align*}
t_{\delta_1} \cdots t_{\delta_{n-1}} t_{\delta_n} t_{c_4} = Q_{10m-12n+42} \cdot t_{c_1}^2 R_{2n-3}. 
\end{align*}
Since $c_1$ and $c_4$ are nonseparating curves and disjoint from each other, there is an element $\Psi_1$ in $\Gamma_2^n$ such that $\Psi_1(c_4)=c_1$ and $\Psi_1(c_1)=c_4$. Therefore, by the relation $t_{\Psi_1(c)}=\Psi_1 t_c \Psi_1^{-1}$, we obtain the following relation:
\begin{align*}
t_{\delta_1} \cdots t_{\delta_{n-1}} t_{\delta_n} t_{c_1} = Q_{10m-12n+42}^\prime \cdot t_{c_4}^2 R_{2n-3}^\prime, 
\end{align*}
where $Q_{10m-12n+42}^\prime = \Psi_1Q_{10m-12n+42}\Psi_1^{-1}$ and $R_{2n-3}^\prime = \Psi_1R_{2n-3}\Psi_1^{-1}$. From the above relations, we have 
\begin{align*}
(t_{\delta_1} \cdots t_{\delta_{n-1}} t_{\delta_n})^k t_{c_4}^{k-1} t_{c_1} &= (t_{\delta_1} \cdots t_{\delta_{n-1}} t_{\delta_n} t_{c_4})^{k-1} (t_{\delta_1} \cdots t_{\delta_{n-1}} t_{\delta_n} t_{c_1}) \\
&= (Q_{10m-12n+42} \cdot t_{c_1}^2 R_{2n-3})^{k-1} (Q_{10m-12n+42}^\prime \cdot t_{c_4}^2 R_{2n-3}^\prime). 
\end{align*}
Since $k=2,3$, we can remove $t_{c_4}^{k-1} t_{c_1}$ from both sides of this relation. Hence, $(t_{\delta_1} \cdots t_{\delta_{n-1}} t_{\delta_n})^k$ can be written as a product of $2(10m-10n+39)$ right-handed Dehn twists about nonseparating curves. The proof for $g\geq 3$ is similar. In this case, we use the relation (\ref{relation4}) and an element $\Psi_2$ in $\Gamma_g^n$ such that $\Psi_2(d^\prime)=c_1$ and $\Psi_2(c_1)=d^\prime$.

Suppose that $k\geq 4$. Since $k=2q+3\epsilon$ for $q\geq 1$ and $\epsilon=0,1$, $(t_{\delta_1} \cdots t_{\delta_n})^k = \{(t_{\delta_1} \cdots t_{\delta_n})^2\}^q (t_{\delta_1} \cdots t_{\delta_n})^{3\epsilon}$ in $\Gamma_g^n$ $(g\geq 2)$ can be written as a product of arbitrarily large number of right-handed Dehn twists about nonseparating curves. 
 
\end{proof}

\begin{remark}
A close look at the proof of Theorem~\ref{Kodaira} makes it evident that whenever we have arbitrarily long positive factorizations of \textit{any} multitwist along boundary curves \ $t_{\delta_1}^{k_1} t_{\delta_2}^{k_2} \cdots t_{\delta_n}^{k_n}$\,, all but finitely many of the corresponding Lefschetz fibrations will be on symplectic $4$-manifolds of general type. In particular, the total spaces of the positive factorizations of the multitwists \ $(t_{\delta_1} t_{\delta_2} \cdots t_{\delta_n})^k$\,, $k \geq 2$, in the above proof should have symplectic Kodaira dimension $\kappa =2$, no matter what $n$ is, which is very different than the case of $k=1$ corresponding to Lefschetz pencils.
\end{remark}

\vspace{0.1in}
\section{Completing the proofs of main theorems and further remarks} \label{Sec:final}

We now bring together various results we have obtained to complete the proofs of Theorems~A, B, C. We will also discuss the length function for further mapping classes, as well as, for its 
restrictions to subgroups of mapping classes, and list a few interesting questions. 

\subsection{Proofs of Theorems~A, B, C} \

To prove our main theorems, we will simply provide navigational guides to the relevant results one needs to assemble, many of which we have obtained in the previous sections.

\begin{proof}[Proof of Theorem~A]
It is well-known that there is a unique genus-$1$ Lefschetz \linebreak \textit{fibration} with $(-1)$-sphere sections, whose total space is $X=E(1)= \CP \# 9 \CPb$. Since $b^-(X)=9$, there are no more than $9$ disjoint $(-1)$-sphere sections in this fibration. It follows that $\h(\Delta)=- \infty$ if $n>9$ and $12$ if $1 \leq n \leq 9$.

All the remaining values of $\h(\Delta)$ are given by Theorem~\ref{Kodaira} and by Theorem~\ref{thm1}.
\end{proof}

\begin{proof}[Proof of Theorem~B]
The mapping class group $\Gamma_{1}^1$ is isomorphic to the braid group on $3$ strands, and it is generated by $t_a,t_b$ for any two nonseparating simple closed curves intersecting at one point. Here $H_1(\Gamma_1^1; \Z) \cong \Z$, generated by any Dehn twist along a nonseparating curve. By the $1$-boundary chain relation, we have  $t_\delta=(t_at_b)^6$ in $\Gamma_g^1$. So for $[t_a]=1$ in $H_1(\Gamma_1^1; \Z)$, we have $[t_\delta]=12$ in $H_1(\Gamma_1^1; \Z)$. 

If $n > 1$, we can cap off all boundary components of $\Sigma_1^n$ but one, which induces a homomorphism from $\Gamma_g^n$ onto $\Gamma_g^1$. Thus, any positive product of $\Delta^k$ in $\Gamma_1^n$, \textit{if exists}, yields a positive product of $t_{\delta}^k$ in $\Gamma_1^1$, which, by the above calculation is equal to $12k$. It follows that $\Delta^k$ has a positive factorization of length $12k + l$ in $\Gamma_g^n$, where $l$ is the number of Dehn twists along curves that separate some of the $n-1$ boundary components we capped off. Since $\h$ is calculated only for nonseparating curves, the latter contribution does not occur, completing the proof of our claim that $\Delta^k$ is precisely $12k$, part~(1) of the theorem. It is a standard fact that any elliptic surface $E(k)$ admits $n \leq 9$ sections of self-intersection $-k$, so $\Delta^k$ admits a positive factorization provided $n \leq 9$.

Part~(2) is covered by Theorem~\ref{thm3} and part~(3) by Theorem~\ref{thm2}. 
\end{proof}

\begin{proof}[Proof of Theorem~C]
In both parts, the value $-\infty$ of $\hT$ or $\h$ is realized by $1 \in \Gamma_g^n$ by Proposition~\ref{nonpositive}, whereas any positive $k$ is realized by $t_c^k$ along any homologically essential curve $c$ by Proposition~\ref{unique}. Note that for $g=0$ and $n=1$, there are no homologically essential curves, and thus no positive factorizations to consider. 

The fact that $\hT(\Gamma_g^n)$ does not contain $+\infty$ under the assumptions in part~(1) follows from Proposition~\ref{smallgenera}. On the contrary, for $g \geq 2$, either by Theorem~\ref{thm2} or Theorem~\ref{thm3}, we have mapping classes in $\Gamma_g^n$ with infinite length, completing the proof of part~(2) of the theorem. 
\end{proof}

\subsection{Further observations and questions} \

As our results demonstrate, knowing that a mapping class admits a positive factorization in the mapping class group of a 
surface (say the page of an open book) does not in general mean that there is an upper bound on the length of all its positive factorizations. The exceptions occur in low genera cases which is essentially due to positive factorizations being lifts of quasipositive braid factorizations, where for the latter, it is known that the degree of a factorization determines the length of all possible factorizations. We can thus ask for which subgroups $N < \Gamma_g^n$, the restriction of $\h$ to $N$, which we denote by $\h _N$, has bounded image. 

Consider the subgroup $\mathcal{H}_g^1$ of $\Gamma_g^1$\,, which consists of mapping classes that commute with a fixed hyperelliptic involution on $\Sigma_g^1$. This group has a non-trivial abelianization, namely $\Z$, which in a similar fashion to our arguments above provides a bound on the length of any factorization into hyperelliptic Dehn twists in $\mathcal{H}_g^1$: the length of any factorization into hyperelliptic Dehn twists along nonseparating curves is fixed. The quotient of $\Sigma_g^1$ under the hyperelliptic involution gives the disk with $2g+1$ marked points. Since any $\Phi \in \mathcal{H}_g^1$ commutes with the hyperelliptic involution, it gives a mapping class of the disk with $2g+1$ marked points. Projecting the branch locus to the quotient then gives a braid in $S^3$, and the class $[f]$ in the abelianization is exactly the writhe of the this braid under the obvious identification of $Ab(\tilde{\Gamma}_g^1)$ with $\Z$. 

Now for $g \geq 3$, let $\Phi= t_{\delta}$ in $\mathcal{H}_g^1$. By the above observation, $\Phi$ has finite length in this subgroup. On the other hand, we have the $1$-boundary component chain relation:
\[ t_{\delta} = (t_{c_1}t_{c_2}\cdots t_{c_{2g}})^{4g+2} \, . \]
It is easy to see that by applying braid relators successively, we get
\[ \Phi= (t_{c_1}t_{c_2}\cdots t_{c_{2g-1}})^{4g+2} \cdot W \, , \]
where $W$ is a positive word that consists of products of conjugates of $t_{c_{2g}}$. By the $2$-boundary chain relation, $(t_{c_1}t_{c_2}\cdots t_{c_{2g}})^{2g+2} = t_{b_1} t_{b_2}$, which is a mapping class with infinite length. It follows that $\Phi$ has infinite length in $\Gamma_g^1$, even though it has finite length in the subgroup $\mathcal{H}_g^1$. We have thus seen:

\begin{proposition}
The image of the positive factorization length function on the subgroup $N=\h_{\mathcal{H}_g^1}$ is strictly smaller  than its image on the \mbox{mapping class group $\Gamma_g^1$.} Namely, $\h_{\mathcal{H}_g^1}(\mathcal{H}_g^1)= \N \cup \{-\infty\}$, whereas $\h(\Gamma_g^1)=\N \cup \{\pm\infty\}$.
\end{proposition}

We therefore see that if the related geometric problem is restrained by positive factorizations in a subgroup of the mapping class group, one can achieve uniform bounds on the topology of the fillings -- which are in addition asked to come from branched coverings o the $4$-ball in the above case. This raises a question that is interesting in its own:

\begin{question}
For which subgroups $N < \Gamma_g^n$ does $\h_N$ have finite, positive image? What is the geometric significance of such $N$?
\end{question}

\vspace{0.3in}

\end{document}